\documentclass[12pt, draft]{amsart}

\usepackage{amstext}
\usepackage{amsthm}
\usepackage{amsmath}
\usepackage{amssymb}
\usepackage{latexsym}
\usepackage{amsfonts}
\usepackage{graphicx}
\usepackage{texdraw}
\usepackage{graphpap}
\usepackage{color}
\usepackage[inline,nomargin]{fixme}

\usepackage[pagebackref,hypertexnames=false, colorlinks, citecolor=black, linkcolor=blue, urlcolor=red]{hyperref}

\usepackage[backrefs]{amsrefs}

\input txdtools

\newcommand{\norm}[1]{\ensuremath{\left\|#1\right\|}}

\newcommand{\D}{\mathbb{D}}

\newcommand{\be}{\begin{equation}}
\newcommand{\ee}{\end{equation}}

\numberwithin{equation}{section}

\newtheorem{thm}{Theorem}[section]
\newtheorem{lm}[thm]{Lemma}

\newtheorem{prop}[thm]{Proposition}
\newtheorem*{prop*}{Proposition}
\newtheorem{question}[thm]{Question}

\theoremstyle{remark}

\newtheorem*{rem*}{Remark}

\begin{document}

\title{Thin sequences and the Gram matrix}

\author[P. Gorkin]{Pamela Gorkin}
\address{Pamela Gorkin, Department of Mathematics\\ Bucknell University\\  Lewisburg, PA  USA 17837}
\email{pgorkin@bucknell.edu}

\author[J.E.McCarthy]{John E. McCarthy}
\address{John E. McCarthy\\
Dept. of Mathematics \\
Washington University in St. Louis\\
St. Louis, MO 63130}
\email{mccarthy@wustl.edu}

\author[S. Pott]{Sandra Pott}
\address{Sandra Pott\\ Faculty of Science\\
Centre for Mathematical Sciences\\
Lund University\\
22100 Lund, Sweden}
\email{sandra@maths.lth.se}

\author[B.D.Wick]{Brett D. Wick}
\address{Brett D. Wick, School of Mathematics\\ Georgia Institute of Technology\\ 686 Cherry Street\\ Atlanta, GA USA 30332-0160}
\email{wick@math.gatech.edu}

\thanks{P.G. Partially supported by Simons Foundation Grant 243653}
\thanks{J.M. Partially supported by National Science Foundation Grant  DMS 1300280}
\thanks{B.D.W Partially supported by National Science Foundation Grant DMS 0955432}

\maketitle

\begin{abstract}
We provide a new proof of Volberg's Theorem characterizing thin interpolating sequences as those for which the Gram matrix associated to the normalized reproducing kernels is a compact perturbation of the identity. In the same paper, Volberg characterized sequences for which the Gram matrix is a compact perturbation of a unitary as well as those for which the Gram matrix is a Schatten-$2$ class perturbation of a unitary operator. We extend this characterization from $2$ to $p$, where $2 \le p \le \infty$.
\end{abstract}

\section{Introduction}

Let $\mathbb{D}$ denote the open unit disk and $\mathbb{T}$ the unit circle. Given $\{\alpha_j\}$, a Blaschke sequence of points in $\D$, we let $B$ denote the corresponding Blaschke product and $B_n$ denote the Blaschke product with the zero $\alpha_n$ removed. Further, we let $\delta_j = |B_j(\alpha_j)|$,  $k_j = \frac{1}{1 - \overline{\alpha_j} z}$ denote the Szeg\H{o} kernel (the reproducing kernel for $H^2$) at $\alpha_j$, $g_j = k_j/\| k_j\|$ the $H^2$-normalized kernel, and $G$ the Gram matrix  with entries $G_{ij} = \langle g_j , g_i \rangle$. In the second part of \cite[Theorem 2]{V}, Volberg's goal was to develop a condition ensuring that $\{g_n\}$ is near an orthogonal basis; by this, one means that there exist $U$ unitary and $K$ compact such that 
$$g_n = (U + K)e_n,$$ where $\{e_n\}$ is the standard orthogonal basis for $\ell^2$. By \cite[Section 3]{F} or \cite[Proposition 3.2]{CFT}, this is equivalent to the Gram matrix defining a bounded operator of the form $I + K$ with $K$ compact. Following Volberg and anticipating the connection to the Schatten-$p$ classes, we call such bases $\mathcal{U} + \mathcal{S}_\infty$ bases. Volberg showed that $\{g_n\}$ is a $\mathcal{U} + \mathcal{S}_\infty$ basis if and only if $\lim_n \delta_n = 1$; in other words, if and only if $\{\alpha_n\}$ is a thin sequence. Assuming $\{g_n\}$ is a $\mathcal{U} + \mathcal{S}_\infty$ basis, it is not difficult to show that the sequence $\{\alpha_n\}$ must be thin. But Volberg's proof of the converse is more difficult and depends on the main lemma of a paper of Axler, Chang and Sarason \cite[Lemma 5]{ACS}, estimating the norm of a certain product of Hankel operators as well as a factorization theorem for Blaschke products. The lemma in \cite{ACS} uses maximal functions and a certain distribution function inequality. A more direct proof of Volberg's result is desirable, and we provide a simpler proof of this result in Theorem~\ref{prop:GMPW} of this paper.

 In a second theorem, letting $\mathcal{S}_2$ denote the class of Hilbert-Schmidt operators, Volberg showed (see \cite[Theorem 3]{V}) that $\{g_n\}$ is a $\mathcal{U} + \mathcal{S}_2$ basis if and only if $\prod_{n = 1}^\infty \delta_n$ converges. We are interested in estimates for the ``in-between'' cases. We provide a new proof of Volberg's theorem for $p = \infty$ and prove the following theorem.

\begin{thm}
For $2 \le p < \infty$, the operator $G - I \in \mathcal{S}_p$ if and only if $\sum_n (1 - \delta_n^2)^{p/2} < \infty$. \end{thm}

Volberg's theorem covered the cases $p = 2$ and $p = \infty$, but our proofs differ in the following way: Instead of using the results of \cite{ACS} and theorems about Hankel operators, we use the relationship between growth estimates of functions that do interpolation on thin sequences (see \cite{E}, \cite{E2}) and the norm of the Gram matrix. This simplifies previous proofs and provides the best estimates available.

\section{Preliminaries and notation}

Let  $\{\alpha_j\}$ be a sequence in $\mathbb{D}$ with corresponding Blaschke product $B$, and $B_j$ be the Blaschke product with zeroes at every point in the sequence except $\alpha_j$ and $\delta_j = |B_j(\alpha_j)|$. 
The separation constant $\delta $ is defined to be $\delta := \inf_j \delta_j$. Carleson's interpolation theorem
says that the sequence $\{\alpha_j\}$ is interpolating if and only if $\delta > 0$, \cite{car58}.
 The sequence $\{\alpha_j\}$ is said to be \textit{thin} if $\lim_{j\to\infty} \delta_j=1$. Given a thin sequence we may arrange the $\delta_j$ in increasing order and rearrange the zeros of the Blaschke product accordingly.

Recall that if $T$ is an operator on a Hilbert space $\mathcal{H}$ and $\lambda_n$ is the $n$th singular value of $T$, then given $p$ with $1 \le p < \infty$ the Schatten-$p$ class, $\mathcal{S}_p$, is defined to be the space of all compact operators with corresponding singular sequence in $\ell^p$, the space of $p$-summable sequences. Then $\mathcal{S}_p$ is a Banach space with norm $$\|T\|_p = \left(\sum |\lambda_n|^p\right)^{1/p}.$$ For $p = \infty$, we let $\mathcal{S}_\infty$ denote the space of compact operators. 



%
%
Recall that $k_j$ denotes the Szeg\H{o} kernel,
$g_j = k_j/\|k_j\|$, and $G$ the Gram matrix  with entries $G_{ij} = \langle g_j , g_i \rangle$.
(The Gram matrix depends of course on the sequence $\{ \alpha_j \}$, but we suppress this in the notation).
For $\{\alpha_j\}$ interpolating, we let $D$ be the diagonal matrix with entries $1/B_j(\alpha_j)$.  It is known (see, for example,  formula (26) of \cite{K}) that
\be
\label{eqa1}
G^{-1} \ = \
D^{*} G^t D .
\ee

For a given sequence $\{ \alpha_j \}$, the interpolation constant is the infimum of those $M$ 
such that for any sequence $\{ a_j \}$ in $\ell^\infty$, one can find a function $f$ in $H^\infty$
with $f(\alpha_j) = a_j$ and $\norm{f}_{\infty}\leq M \norm{a}_{\ell^\infty}$.
We shall let $M(\delta)$ denote the supremum of the interpolation constants over all
sequences $\{ \alpha_j \}$ with separation constant $\delta$.

The following  result is due essentially to A. Shields and H. Shapiro
\cite{shsh}. See
\cite[Proposition 9.5]{AM} for a proof of this version.
%
%
%

\begin{prop}\label{propb1}
Let $\{\alpha_j\}$ be an interpolating sequence in $\mathbb{D}$.

 (i) If the interpolation constant is $M$, then both $\|G \|$ and $\| G^{-1} \|$ are bounded by
 $M^2$.
 
(ii) If $\| G \| = C_1$ and $\| G^{-1} \| = C_2$, then the interpolation constant is
bounded by $\sqrt{C_1 C_2}$.
\end{prop}

We shall use the following estimate of J.P. Earl (see \cite{E} or \cite{E2})
 to obtain our results. 
\begin{thm}[Earl's Theorem] \label{thm:E}


The interpolation constant $M(\delta)$ 
satisfies
$$M(\delta) \le \left(\frac{1 + \sqrt{1 - \delta^2}}{\delta}\right)^2.$$
\end{thm}
%
%

\section{Schatten-$p$ classes}

In this section we provide estimates on the Schatten-$p$ norm of $G - I$. We will need the theorem and lemma below.

\begin{thm}\label{zhu1}(see e.g. \cite[Theorem 1.33]{Z}) Let $T$ be an operator on a separable Hilbert space, $\mathcal{H}$.

If $0<p\leq2$ then
$$
\norm{T}_{\mathcal{S}_p}^p=\inf\left\{\sum_{n} \left\Vert Te_n \right\Vert^p:\{e_n\} \textnormal{ is any orthonormal basis in } \mathcal{H} \right\}
$$
and if $2\leq p<\infty$
$$
\norm{T}_{\mathcal{S}_p}^p=\sup\left\{\sum_{n} \left\Vert Te_n \right\Vert^p:\{e_n\} \textnormal{ is any orthonormal basis in } \mathcal{H} \right\}.
$$

\end{thm}

We say that two sequences $\{x_n\}$ and $\{y_n\}$ of positive numbers are equivalent if there exist constants $c$ and $C$, independent of $n$, such that $c y_n \le x_n \le C y_n$ for all $n$. We write $x_n \asymp y_n$.  We will also write $A\lesssim B$ to indicate that there exists a constant $C$ such that $A\leq C B$.

\begin{lm} \label{4.1}
Let $\{e_n\}$ denote the standard orthonormal basis for $\ell^2$ and $\{\alpha_j\}$ be an interpolating sequence in $\mathbb{D}$ with corresponding $\delta_j$. Then
\[\|(G - I) e_n\| \asymp \sqrt{1 - \delta_n^2}.\]
\end{lm}

\begin{proof}
We have
\begin{eqnarray*}
\|(G - I)e_n\|^2 & = & \langle (G^\ast - I)(G - I) e_n, e_n \rangle\\
& = & \Bigg{\langle}\begin{pmatrix}\langle g_n, g_1\rangle\\\ldots\\\langle g_n, g_{n - 1}\rangle\\0\\\langle g_n, g_{n + 1}\rangle\\\ldots\end{pmatrix},\begin{pmatrix}\langle g_n, g_1\rangle\\\ldots\\\langle g_n, g_{n - 1}\rangle\\0\\\langle g_n, g_{n + 1}\rangle\\\ldots\end{pmatrix}\Bigg{\rangle}\\
& = & \sum_{j \ne n} \langle g_n, g_j \rangle \overline{\langle g_n, g_j \rangle}\\
& = & \sum_{j \ne n} \left|\frac{\sqrt{1 - |\alpha_j|^2}\sqrt{1 - |\alpha_n|^2}}{1 - \overline{\alpha_n} \alpha_j}\right|^2\\
& = &  \sum_{j \ne n} 1 - \left|\frac{\alpha_j - \alpha_n}{1 - \overline{\alpha_n} \alpha_j}\right|^2.
\end{eqnarray*}
But  $-\log x \ge 1 - x$ for $x > 0$ and $-\log x < c (1 - x)$ for $x < 1$ bounded away from $0$ and some constant $c$ independent of $x$, so $-\log x \asymp 1 - x$ for $x$ bounded away from $0$. Consequently,
\begin{eqnarray*}
\|(G - I)e_n\|^2
& \asymp &   \sum_{j \ne n} -\log\left|\frac{\alpha_j - \alpha_n}{1 - \overline{\alpha_n} \alpha_j}\right|^2\\
& = &   -\log \prod_{j \ne n} \left|\frac{\alpha_j - \alpha_n}{1 - \overline{\alpha_n}\alpha_j}\right|^2\\
& = & -\log \delta_n^2\\
& \asymp &   1 - \delta_n^2.
\end{eqnarray*}

Note that the constants involved do not depend on $n$.
\end{proof}
 
Combining the lemma with Theorem~\ref{zhu1}, we obtain the following theorem.

\begin{thm}\label{thm:estimate}
The following estimates hold:
\begin{itemize}
\item If $2\leq p<\infty$ then 
$$
\sum_{n} (1-\delta_n)^{\frac{p}{2}}\lesssim \norm{G-I}_{\mathcal{S}_p}^p;
$$
\item If $0<p\leq 2$ then 
$$
\norm{G-I}_{\mathcal{S}_p}^p\lesssim \sum_{n} (1-\delta_n)^{\frac{p}{2}};
$$
\item If $p=2$ then 
$$
\sum_{n} (1-\delta_n)\asymp \norm{G-I}_{\mathcal{S}_2}^2. 
$$
\end{itemize}
\end{thm}

\begin{lm}\label{lm:useful2}
Let $\{\alpha_j\}$ be an interpolating sequence and $G$ the corresponding Gram matrix.
Let $C = \| G^{-1}\|$.
 Then  $\|G - I\| \le C-1$.
\end{lm}

\begin{proof}
By (\ref{eqa1}), we have $G \leq C I$, and as $G$ is a positive operator,
we have $G \geq (1/C) I$. Therefore
\[
\left(\frac{1}{C} - 1 \right)I \ \leq \ G-I \ \leq \ (C -1) I ,
\]
and as $C > 1$, we get  $\|G - I\| \le C-1$.
\end{proof}

In what follows, for a positive integer $N$, we let $G_N$ denote the lower right-hand corner of the Gram matrix obtained by deleting the first $N$ rows and columns of $G$. Thus,\\

\[G_N = \left( \begin{array}{cccccc}
1 & \langle g_{N+1}, g_N\rangle & \cdots &\langle g_{N+j}, g_N\rangle & \cdots \\
\langle g_{N}, g_{N + 1}\rangle & 1 & \cdots & \langle g_{N + j}, g_{N + 1}\rangle & \cdots\\
\cdots & \cdots & \cdots & \cdots & \cdots\\
\end{array} \right)\]\\
and $\lambda_n \ge 0$ denotes the $n$-th singular value of $G-I$, where the singular values are arranged in decreasing order.

We are now ready to provide our simpler proof of Volberg's result \cite[Theorem 2, p. 215]{V}. 

\begin{thm}
\label{prop:GMPW}
The sequence $\{\alpha_n \}$ is a thin 
 sequence if and only if
the Gram matrix $G$ is the identity plus a compact operator.
\end{thm}
\begin{proof} 
($\Rightarrow$)
Suppose $\{\alpha_n \}$ is thin. By discarding finitely many points in the sequence $( \alpha_n )$, we can assume that the
sequence has a positive separation constant, and hence is interpolating.

Let $G_N$ be the Gram matrix of $\{ g_j \}$ for $j \geq N$. 
We shall let  $\delta_{N, j}$ denote the $\delta_j$ defined for $G_N$ (that is, corresponding to the Blaschke sequence $\{\alpha_j\}_{j \ge N}$). 
Note that  $\delta_{N, j} \ge \delta_{N + j}$ for $j = 0, 1, 2, \ldots$, and so we have that $\delta(N) :=\inf_j \delta_{N, j} \ge \inf_j \delta_{N + j} = \delta_N$. By 
Theorem~\ref{thm:E} and
Proposition~\ref{propb1}, 
$$
\|G_N^{-1}\| \  \le \ (M(\delta(N))^2 \ \le  \ \frac{\left(1 + \sqrt{1 - \delta(N)}\right)^4}{\delta(N)^4} \ \le \frac{\left(1 + \sqrt{1 - \delta_N^2}\right)^4}{\delta_N^4}. 
$$
Applying Lemma~\ref{lm:useful2}
$$\|G_N - I_N\| \le \left(\frac{\left(1 + \sqrt{1 - \delta_N^2}\right)^4}{\delta_N^4} - 1\right) \leq C \sqrt{1 - \delta_N},$$ where $C$ is a constant independent of $N$. Since $\sqrt{1 - \delta_N} \to 0$ as $N \to \infty$, we conclude that $G-I$ is compact.

($\Leftarrow$) From \eqref{eqa1}, we have
\be
\label{eqc1}
G^{-1} - I = D^* (G^t - I) D + [ D^* D - I ] .
\ee
If $G-I$ is compact, then so are $G^t - I$ and $G^{-1} - I = G^{-1}(I - G)$.
Therefore from \eqref{eqc1}, we have
$D^*D - I$ is compact, which means $\lim_{j \to \infty}\delta_j^2 = 1$. Consequently, the sequence is thin.
\end{proof}

\begin{thm}
\label{thm:pge2}
For $2 \le p < \infty$, the operator $G - I \in \mathcal{S}_p$ if and only if $\sum_n (1 - \delta_n^2)^{p/2} < \infty$. \end{thm}

\begin{proof}
By Theorem~\ref{thm:estimate}, if $G - I \in \mathcal{S}_p$, then the sum is finite.

Now suppose the sum is finite. Using Lemma~\ref{lm:useful2} as in Theorem~\ref{prop:GMPW}, we have

$$\|G_N - I_N\| \le C \sqrt{1 - \delta_N},$$ where $C$ is independent of $N$.

By \cite[Theorem 1.4.11]{Z},
$$|\lambda_{N + 1}| \le \inf\{\|(G - I) - F\|: F \in \mathcal{F}_N\},$$ where $\mathcal{F}_N$ is the set of all operators of rank less than or equal to $N$. Therefore, taking $F$ to be the matrix with the same first $N$ rows and columns as $G - I$, which is of rank at most $2N$, 
we have $$|\lambda_{2N+1}| \le \|G_{N+1} - I_{N+1}\| \le C \sqrt{1 - \delta_{N+1}},$$ by our computation above. Therefore $$|\lambda_{2N+1}|^p \le C^p (1 - \delta_{N+1})^{p/2}.$$ Since the singular values are arranged in decreasing order, $|\lambda_{2n + 1}| > |\lambda_{2n}|$ for each $n$. Thus, if $\sum_N (1 - \delta_N)^{p/2} < \infty$, then $\sum_n |\lambda_{2n}|^p \le  2 \sum_n |\lambda_{2n+1}|^p < \infty$   and  we conclude that $G - I \in \mathcal{S}_p$.

\end{proof}

We conclude by remarking that it is possible to trace through the proofs above to determine constants $c$ and $C$ ,
which depend only on $\delta = \inf_n {\delta_n}$,
such that for $ 2 \leq p \leq \infty$,

\be
\label{eq:const}
c \| \sqrt{1-\delta_n} \|_{\ell^p} \ \le \
\norm{G-I}_{\mathcal{S}_p} \ \le  \ 
C \| \sqrt{1-\delta_n} \|_{\ell^p} .
\ee
In particular, by choosing $\delta$ close enough to $1$, one can choose 
$c$ and $C$ in \eqref{eq:const} arbitrarily close to
$\sqrt{2}$ and $4 \sqrt{2} (2^{1/p})$, respectively.

\begin{question}
Is Theorem~\ref{thm:pge2} true for $p < 2$?
\end{question}

{\bf Acknowledgement.} We thank the referee for his or her careful reading of this manuscript as well as helpful comments.

\end{document}